\documentclass[a4paper,10pt]{amsart}

\usepackage[utf8]{inputenc}
\usepackage[top=4.5cm,bottom=4.25cm,right=3.9cm,left=3.9cm]{geometry}
\usepackage{amsmath,amssymb,amsthm}
\usepackage{url}
\usepackage[pagebackref = true]{hyperref}
\hypersetup{
colorlinks = true,
pdfstartview = FitH,
pdfpagemode = UseNone,
pdfauthor    = {Giacomo Cherubini, Pavlo Yatsyna},
pdftitle     = {Cyclic polytope of the simplest cubic fields},
pdfkeywords  = {cyclic polytope, cubic fields, lattice counting},
pdfcreator   = {Pdflatex},
linktoc = none
}

\numberwithin{equation}{section}

\def\eps{\epsilon}
\def\ds{\displaystyle}

\def\ZZ{\mathbb{Z}}
\def\RR{\mathbb{R}}

\theoremstyle{plain}
\newtheorem{theorem}{Theorem}[section]
\newtheorem{lemma}[theorem]{Lemma}
\newtheorem{corollary}[theorem]{Corollary}

\title{Cyclic polytope of the simplest cubic fields}
\author{Giacomo Cherubini}
\author{Pavlo Yatsyna}

\address{
         Charles University,
         Faculty of Mathematics and Physics,
         Department of Algebra,
         Sokolov\-sk\'a 83, 18600 Praha~8,
         Czech Republic
        }

\email{
    cherubini@karlin.mff.cuni.cz\\
    yatsyna@karlin.mff.cuni.cz
    }

\date{\today}
\begin{document}

%% ABSTRACT %%
\begin{abstract}
In this paper, we study dilation of cyclic polytopes with the vertices defined by a generator of the simplest cubic fields. In particular, for a specific range of values, we give a precise number of the contained lattice points. 
\end{abstract}

%% TITLE %%
\maketitle

%% INTRODUCTION %%
\section{Introduction}
Lattice point enumeration is a much-studied problem, especially in number theory,
with the notable example being the Gauss circle problem
and various variants thereof (see for example \cite{Hux}).
For a given object in $\RR^d$, its volume is the first `good' estimate for the number of lattice points it contains. If a polytope $P$ is rational, that is, the convex hull of rational vertices, Ehrhart~\cite{ehrhart} showed that $|kP\cap\ZZ^d|$ is a quasi-polynomial \cite[Ch. 3]{BR}. No such formula is known for irrational polytopes (when at least one vertex is not rational). Yet, one can study the discrepancy between the number of lattice points and the volume of the polytope, and a considerable amount of work has been done in this direction \cite{beck,HL,matousek}. One pertinent example, if $P\subset \RR^2$ has irrational algebraic vertices, then for any $\epsilon>0$ we have $|\mathrm{area}(P)-|kP\cap \ZZ^2||\ll k^{\epsilon}$ \cite[Theorem 4]{skr1}.

We shall focus on irrational algebraic polytopes, specifically, cyclic polytopes whose vertices correspond to a generator of the simplest cubic fields. Given a moment curve $\psi_d:\RR\longrightarrow \RR^d$ which maps
$r\mapsto (r,r^2,\ldots,r^d)$, the cyclic polytope on $n$ points
is the polytope whose vertices are $\psi_d(\rho_1),\ldots,\psi_d(\rho_n)$ for some $\rho_i\in\RR$ \cite[p. 97]{matousek}. In our case, we take $\rho_i$ as the roots of the polynomial:
\begin{equation}\label{f_a}
f_a(x) = x^3-ax^2-(a+3)x-1,
\end{equation}
for
$a\ge 1$, $a\in \ZZ$.
If we denote by $\mathcal{T}_a$ the cyclic polytope with vertices
$\psi_2(\rho_i)$, we show that the difference between $k\mathcal{T}_a\cap\ZZ^2$
and the volume of $k\mathcal{T}_a$ exhibits a regular behaviour for $k\leq a$.

%% MAIN THEOREM %%
\begin{theorem}\label{thm2}
Let $a\geq 1$ and $k\leq a$.
Let $\mathcal{T}_a$ be the cyclic polytope with vertices $\psi_2(\rho_i)$,
where $\psi_2$ and $\rho_i$ are as above. Then
\[
|k\mathcal{T}_a\cap \ZZ^2| = \frac{k^2(a^2+3a+9)}{2} + \theta(k),
\quad
\theta(k)=
\begin{cases}
1 & k\text{ even},\\
-1/2 & k \text{ odd}.
\end{cases}
\]
\end{theorem}

The simplest cubic fields are the splitting fields of $f_a$ and were introduced by Shanks in \cite{shanks}
with the goal of finding number fields with a large class number.
Since then, this family became a ``testing ground'' for many investigations of cubic number
fields \cite{KM,LP,wash}.
They have several favorable properties: the discriminant is a perfect square,
expressed by $\Delta=(a^2+3a+9)^2$; the field is monogenic for infinitely many values of $a$;
the roots of $f_a$ are given by a simple explicit formula (see \eqref{exactformula})
and converge to rational integers when $a$ is large; and the regulator is explicitly known
(and relatively small, thus giving a large class number).
In relation to the cyclic polytope $\mathcal{T}_a$,
we mention that its volume is given by
\[
\mathrm{vol}(\mathcal{T}_a) = \frac{1}{2}\sqrt{\Delta} = \frac{a^2+3a+9}{2}
\]
and therefore $\theta(k)$ in Theorem \ref{thm2} measures the discrepancy between the number of lattice points
and the volume of $k\mathcal{T}_a$, as desired.

For $k=1,2$, Theorem \ref{thm2} was already known (\cite[Exemple 1]{cohen} and \cite[Theorem~4.4]{KK})
primarily for applications to finding the value of the Dedekind zeta function evaluated at negative integers.
Note that for $a\not\equiv 0\pmod{3}$, $f_a$ is equivalent to the polynomial in Cohen's Exemple 1 in \cite{cohen}:
$a \equiv \pm 1 \pmod{3}$, then letting $e=a^2+3a+9$ and $u=\pm 2a\pm 3$ we get
$\pm f_a(\pm x+\frac{a\mp 1}{3})$.

The number $|P\cap\ZZ^2|$ can be used to bound the ranks of the universal quadratic forms defined over the number field \cite{pavlo}, yet, the knowledge of $|kP\cap \ZZ^2|$ can be applied to sharpen those bounds (as recently was applied in \cite{KM}).

The cyclic polytopes corresponding to algebraic integers are not well studied. It is know that if such polytopes contain at least a single lattice point, then their volume is bounded from below, but there exist empty polytopes exceeding this bound. Furthermore, there is an infinite family of polynomials for which the corresponding cyclic polytope is empty. In particular, let $\zeta_n$ be the $n$th root of unity. Then for $n=p^l,$ where $p$ is an odd prime number, the cyclic polytope related to $\zeta_n+\zeta^{-1}_n$ contains exactly $\frac{p-1}{2}$ lattice points \cite[Theorem 1.1]{MS}. It is a polytope in dimension $\frac{\phi(n)}{2}$. On the other hand, for a square-free $n>20$, $n\not=30,~p$ or  $2p$, then the cyclic polytope related to $\zeta_n+\zeta^{-1}_n$ is empty \cite[Theorem 2]{pavlo}. In general, for non-prime dimensions, it is not known whether the number of empty (nonequivalent) polytopes is finite. But for dimensions $2,3,4,5$ and $6$, polytopes corresponding to irreducible monic polynomials are not pointless (see \cite{pavlo} for all the above).  

To prove Theorem \ref{thm2} we apply the approach of Skriganov \cite[\S 6]{skr1},
but gain slight improvements by using more precise bounds for our case, see Lemma \ref{lemma:bounds}.

In the case of real quadratic fields, for example when $\mathbb{Q}(\sqrt{D})$ with $D\equiv 2,3 \pmod{4}$,
we get two roots of the associated polynomial, say $\alpha_{1,2}=m\pm n\sqrt{D}$, where $m,n \in\ZZ$.
%and $\{1,\omega\}$ are generators of the ring of integers.
The corresponding cyclic polytope (of dimension one)
becomes a line $\mathcal{L}$ between $\alpha_1$ and~$\alpha_2$, and counting lattice points in $k\mathcal{L}$
reduces to studying
the homogeneous Beatty sequence $\lfloor kn\sqrt{D}\rfloor$.
%$|[k\alpha_1,k\alpha_2]\cap \ZZ|=2\lfloor kn\sqrt{D}\rfloor+1$,
%i.e.~a homogeneous Beatty
%\todo{by having the $2$ out and by shifting by $1$ is it still a Beatty sequence?}
%sequence.

In closing, we mention that Theorem \ref{thm2} can be expressed in terms of
counting integral polynomials interlacing $f_a$ (see \cite[Definition 2.1]{pavlo} for the definition of interlacing).
We state this application as a corollary.
\begin{corollary}
Let $a\geq 1$ and $k\leq a$. Then the number of polynomials $g(x)=kx^2+n_1x+n_2\in\ZZ[x]$
interlacing $f_a$ is given by the quantity on the right in Theorem~\ref{thm2}.
\end{corollary}

%% ACKNOWLEDGMENT %%
\subsection*{Acknowledgment}
This work was supported by the project PRIMUS/20/SCI/002 from Charles University.

%% SECTION 2: OUTLINE OF THE PROOF %%
\section{Outline of the proof}

In order to prove Theorem \ref{thm2}
we observe that counting the points in $k\mathcal{T}_a$
can be done by slicing the triangle along vertical lines
corresponding to points $(x,y)\in\ZZ^2$ with
$x$ lying above a fixed integer $n$,
counting the integers points of $k\mathcal{T}_a$ on each line
and then summing over all lines.
Following this strategy, we will count
carefully the points in
$k\mathcal{T}_a\cap\ZZ^2$ along every single vertical line.

%% SUBSECTION: APPROXIMATION OF THE ROOTS %%
\subsection{Approximation of the roots}
A key input in our proof is a good estimate on
$\rho_1,\rho_2,\rho_3$.
A quick numerical check shows that $\rho_1\approx -1$,
$\rho_2\approx 0$ and $\rho_3\approx a+1$.
This is true in general:
in \cite[(8)]{shanks} Shanks showed that,
on setting $P=a^2+3a+9$, we have the exact formula
\begin{equation}\label{exactformula}
\rho_3 = \frac{1}{3}(a+\sqrt{P}\cos\varphi),
\quad
\varphi = \frac{1}{3}\arctan\frac{\sqrt{27}}{2a+3},
\end{equation}
from which one can deduce that for large values of $a$ we have
\begin{equation}\label{2010:eq001}
\rho_3 = a+1+\frac{4}{2a+3} + O\left(\frac{1}{a^3}\right).
\end{equation}
The location of the other roots follows
from \eqref{2010:eq001} and the relation (see e.g.~\cite[(5)]{shanks})
\begin{equation}\label{relation}
\rho_1 = -\frac{1+\rho_3}{\rho_3}
\quad\text{and}\quad
\rho_2 = -\frac{1}{1+\rho_3}.
\end{equation}
In particular we see that $\rho_1\in(-2,-1)$
and $\rho_2\in(-1,0)$, at least for large $a$.
We make the approximation in \eqref{2010:eq001} more precise by proving the following
result in which we give explicit upper and lower bounds on the error.

%% LEMMA: BOUNDS ON rho_3 %%
\begin{lemma}\label{lemma:bounds}
Let $a\geq 34$.
Let $f_a$ be as in \eqref{f_a} and let $\rho_3$ be its largest root. Then
\begin{equation}\label{rho3bounds}
\rho_3 = a+1+\delta,
\quad\text{with}\quad
\frac{4}{2a+3} -\frac{7}{2a^3} < \delta < \frac{4}{2a+3}-\frac{3}{a^3}.
\end{equation}
\end{lemma}
\begin{proof}
It is easy to verify that for $a\geq 34$ we have
\[
f_a\left(a+1+\frac{4}{2a+3}-\frac{7}{2a^3}\right)
<0<
f_a\left(a+1+\frac{4}{2a+3}-\frac{3}{a^3}\right).
\]
The result follows from the continuity of $f_a$.
\end{proof}

Since the above lemma requires $a\geq 34$, we will prove Theorem \ref{thm2}
under this assumption. A computer check confirms
that the result holds for $1\leq a\leq 33$.

For $i=1,2,3$, let us denote by $P_i$
the vertices of $\mathcal{T}_a$, that is, $P_i=\psi_2(\rho_i)=(\rho_i,\rho_i^2)$. Also, let $\ell_{ij}$ be the
line connecting $kP_i$ and $kP_j$.
By a direct computation we can write
\begin{equation}
\ell_{ij}(x) = (\rho_i+\rho_j)x-k\rho_i\rho_j.
\end{equation}
Notice that by construction the triangle $k\mathcal{T}_a$
lies in the upper half of the plane $\RR^2$.
The above discussion about the roots shows that we have
\begin{equation}\label{floordiff}
|k\mathcal{T}_a\cap \ZZ^2\cap\{x=n\}|
=
\begin{cases}
\lfloor\ell_{13}(n)\rfloor - \lfloor\ell_{12}(n)\rfloor & \text{if }n\leq -1\\
\lfloor\ell_{13}(n)\rfloor - \lfloor\ell_{23}(n)\rfloor & \text{if }n\geq 0,\\
\end{cases}
\end{equation}
for each integer~$n$ such that the intersection is non-empty.
In what follows we show an exact formula for $\lfloor\ell_{ij}(n)\rfloor$ and obtain Theorem \ref{thm2}
in Section \ref{S3} by using \eqref{floordiff} and by summing over all $n$ such that $k\mathcal{T}_a$ has a point lying above $n$.
As a first step, let us use Lemma \ref{lemma:bounds}
to determine for which $n$ the line $\{x=n\}$ intersects $k\mathcal{T}_a$.

%% LEMMA: RANGE FOR n %%
\begin{lemma}
Let $a\geq 34$, $k\leq a$ and $(x,y)\in k\mathcal{T}_a\cap \ZZ^2$.
Then $-k\leq x\leq k(a+1)+2$. If $k\leq\frac{a}{2}$ then $x\leq k(a+1)+1$.
\end{lemma}
\begin{proof}
Write $b=1,2$ according to whether $k\leq \frac{a}{2}$ or $\frac{a}{2}<k\leq a$.
It suffices to show
\begin{equation}\label{2509:eq001}
-k-1<k\rho_1<-k
\quad\text{and}\quad
k(a+1)<k\rho_3<k(a+1)+b.
\end{equation}
First, by Lemma \ref{lemma:bounds} we immediately see that $\rho_3>a+1$ and $\rho_1<-1$,
therefore $k\rho_3>k(a+1)$ and $k\rho_1<-k$, proving two of the inequalities in \eqref{2509:eq001}.
Concerning the first inequality in \eqref{2509:eq001}, again by Lemma \ref{lemma:bounds} we have
\[
-k-1<k\rho_1\iff -1<-\frac{k}{a+1+\delta}\iff k<a+1+\delta,
\]
which holds since $\delta>0$. Finally,
we can write by \eqref{rho3bounds}
\[
k\rho_3<k(a+1)+b
\iff
k\delta < b
\impliedby
k \leq \frac{(2a+3)b}{4},
\]
which proves the last inequality in \eqref{2509:eq001}.
\end{proof}

%% SUBSECTION : TOP LINE %%
\subsection{Top line}
As we have explained, the triangle $k\mathcal{T}_a$ lies in the upper half of $\RR^2$.
More precisely, it is composed of a top edge
which crosses the two quadrants in the upper half plane,
an edge fully contained in the second quadrant ($x<0<y$),
and an edge almost entirely contained~in~the first quadrant ($x,y>0$).
We start by discussing the edge on the top of the triangle,
corresponding to the line $\ell_{13}$.

%% LEMMA: TOP LINE EVALUATION %%
\begin{lemma}\label{lemma:topline}
Let $a\geq 34$ and $k\leq a$. Define the quantities
\[
g = a+2,
\quad
C = kg,
\quad
n_* = a+1-2k,
\quad
n^* = n_*+g\lfloor k/2\rfloor +1.
\]
Then we have
\[
\lfloor\ell_{13}(n)\rfloor
=
\begin{cases}
an + C & -k\leq n \leq n_{*}\\
an + C + i & (i-1)g + 1 \leq n-n_* \leq ig,\; 1\leq i\leq\lfloor k/2\rfloor\\
an + C + \lfloor k/2\rfloor & n=n^{*}\\
an + C + \lfloor k/2\rfloor + j & (j-1)g + 1 \leq n-n^* \leq jg,\;1\leq j\leq\lceil k/2\rceil-1\\
an + C + k & n^{*}+(\lceil k/2\rceil-1)g+1\leq n\leq k(a+1)+\mathbf{1}_{\{k>a/2\}}.
\end{cases}
\]
\end{lemma}

\begin{proof}
By Lemma \ref{lemma:bounds}
and recalling that $\rho_1+\rho_2+\rho_3=a$
and $\rho_1\rho_2\rho_3=1$,
we can write
$\ell_{13}(n) = an+C + \xi$, where
\[
\xi = \xi(a,n,k) = \frac{n}{a+2+\delta} + k\delta
\]
and $\delta$ is as in \eqref{rho3bounds}.
Denote by $\delta_{\pm}$ the
endpoints in the range for $\delta$ in \eqref{rho3bounds},
so that $\delta\in(\delta_{-},\delta_{+})$.
Let $-k\leq n\leq n_*$ and let us prove that $\xi\in(0,1)$.
By \eqref{relation} we have
\begin{equation}\label{0510:eq001}
0<\xi<1
\impliedby
\begin{cases}
k(\delta(a+2+\delta)-1)>0,\\
\ds k<\frac{1+\delta}{\delta(a+2+\delta)-2}.\rule{0pt}{18pt}
\end{cases}
\end{equation}
Since $\delta(a+2+\delta)>2$ for every $\delta\in(\delta_{-},\delta_{+})$
and the fraction in the second line
is greater than $a$ for all $\delta\in(\delta_{-},\delta_{+})$,
both inequalities hold for $k\leq a$.

Consider now $n\in[n_*+(i-1)g+1,n_*+ig]$, and let us prove that $\xi\in(i,i+1)$
for every $i=1,\ldots,\lfloor k/2\rfloor$. We have
\[
\xi\in(i,i+1)
\impliedby
\begin{cases}
k(\delta+4-2\delta(a+2+\delta))<0,\\
\ds k<\frac{1+2\delta}{\delta(a+2+\delta)-2}.\rule{0pt}{18pt}
\end{cases}
\]
For all $\delta\in(\delta_{-},\delta_{+})$
the coefficient of $k$ in the first line is negative
and the fraction in the second line
is greater than $a$, so both inequalities hold.
When $n=n^{*}$ we have
\begin{equation}\label{0510:eq003}
\xi\in(\lfloor k/2\rfloor,\lfloor k/2\rfloor+1)
\impliedby
\begin{cases}
k(2\delta(a+2+\delta)-4-\delta) > -2a-4,\\
\ds k < \frac{\delta}{2\delta(a+2+\delta)-4-\delta}.\rule{0pt}{18pt}
\end{cases}
\end{equation}
Since $2\delta(a+2+\delta)-4-\delta>0$
for all $\delta\in(\delta_{-},\delta_{+})$,
the first inequality holds.
As for the second one,
we verify that the fraction is greater than $a$ for all $\delta\in(\delta_{-},\delta_{+})$.

Consider now $n\in[n^{*}+(j-1)g+1,n^{*}+jg]$,
for $j=1,\ldots,\lceil k/2\rceil-1$. In these intervals we have
\[
\xi\in(\lfloor k/2\rfloor+j,\lfloor k/2\rfloor+j+1)
\impliedby
\begin{cases}
\ds k<\frac{1+\delta}{2+\delta-\delta(a+2+\delta)},\\
\ds k<\frac{3\delta}{2\delta(a+2+\delta)-4-\delta}.\rule{0pt}{18pt}
\end{cases}
\]
The second inequality follows from the second bound in \eqref{0510:eq003}.
In the first line, the fraction greater than $a$ for all $\delta\in(\delta_{-},\delta_{+})$.

Finally consider the last range, $n\in[n^{*}+g(\lceil k/2\rceil-1)+1,ka+k+\mathbf{1}_{\{k>a/2\}}]$
and let us prove that $\xi\in(k,k+1)$. We have
\[
\xi\in(k,k+1)
\impliedby
\begin{cases}
\ds k< \frac{1}{2+\delta-\delta(a+2+\delta)},\\
\ds k<\frac{a+1+\delta}{\delta(a+2+\delta)-1-\delta}\rule{0pt}{18pt}.
\end{cases}
\]
In both inequalities the right-hand side is greater than $a$ for all $\delta\in(\delta_{-},\delta_{+})$, from which the result follows.
\end{proof}

%% SUBSECTION : LEFT LINE %%
\subsection{Left line}
The edge of the triangle $k\mathcal{T}_a$ on the left of the vertical axis is a segment that goes roughly from the point $kP_1\approx(-k,k)$ to the point $kP_2\approx(0,0)$. Because of this, we expect that $\ell_{12}(n)\approx -n$. More precisely, we have the following.

%% LEMMA: LEFT LINE EVALUATION %%
\begin{lemma}\label{lemma:leftline}
Lt $a\geq 34$ and $k\leq a$. Then we have
\[
\lfloor\ell_{12}(n)\rfloor
=
\begin{cases}
-n & -k\leq n\leq -\lfloor k/2\rfloor -1\\
-n-1 & -\lfloor k/2\rfloor \leq n\leq -1.
\end{cases}
\]
\end{lemma}
\begin{proof}
By Lemma \ref{lemma:bounds}
and recalling that $\rho_1+\rho_2+\rho_3=a$
and $\rho_1\rho_2\rho_3=1$,
we can write
$\ell_{12}(n) = -n+ \xi$, where
\[
\xi = \xi(a,n,k) = -n\delta-\frac{k}{a+1+\delta}.
\]
Here $\delta\in(\delta_{-},\delta_{+})$,
with $\delta_{\pm}$ given in \eqref{rho3bounds}.
Consider first $n\in[-k,\ldots,-\lfloor k/2\rfloor-1]$ and let us show that $\xi\in(0,1)$.
We have
\[
\xi\in(0,1)
\impliedby
\begin{cases}
\ds k<\frac{\delta(a+1+\delta)}{2-\delta(a+1+\delta)},\\
k(2-\delta(a+1+\delta))>-(2+\delta)(a+1+\delta).\rule{0pt}{15pt}
\end{cases}
\]
The second inequality is trivially satisfied since $2-\delta(a+1+\delta)$ is positive for all $\delta\in(\delta_{-},\delta_{+})$.
The fraction in the first line is greater than $a$ for every $\delta\in(\delta_{-},\delta_{+})$,
hence the first inequality is satisfied as well.

Consider now $n\in[-\lfloor k/2\rfloor,-1]$
and let us show that in this range we have $\xi\in(-1,0)$.
We can write
\[
\xi\in(-1,0)
\impliedby
\begin{cases}
k<(1+\delta)(a+1+\delta),\\
k(2-\delta(a+1+\delta))>-\delta(a+1+\delta).
\end{cases}
\]
Since $k\leq a$ by assumption, the first inequality holds. In the second line, since the coefficient of $k$ is positive for every $\delta\in(\delta_{-},\delta_{+})$,
we see that the left-hand side is positive while the right-hand side is negative, so the inequality holds trivially.
\end{proof}

%% SUBSECTION : RIGHT LINE %%
\subsection{Right line}
Now we consider the edge of the triangle $k\mathcal{T}_a$
associated to the line~$\ell_{23}$. Such a line connects the points
$kP_2\approx (0,0)$ and $kP_3\approx(k(a+1),k(a+1)^2)$.
Because of this, we expect that $\ell_{23}(x)\approx (a+1)x$.
While the slope is correct, the constant term needs to be adjusted
as follows.

%% LEMMA: RIGHT LINE EVALUATION %%
\begin{lemma}\label{lemma:rightline}
Let $a\geq 34$ and $k\leq a$.
Define the quantities
\[
g = a+1,
\quad
n_{*} = k-1,
\quad
n^{*} = n_{*}+g\lceil k/2\rceil+1.
\]
Then we have
\[
\lfloor\ell_{23}(n)\rfloor
=
\begin{cases}
gn + n_{*} & 0\leq n\leq n_{*}\\
gn+ n_{*}+i & (i-1)g+1\leq n-n_{*}\leq ig,\; 1\leq i\leq \lceil k/2\rceil\\
gn+n_{*}+\lceil k/2\rceil & n=n^{*}\\
gn+n_{*}+\lceil k/2\rceil +j & (j-1)g+1\leq n-n^{*}\leq +jg,\; 1\leq j\leq \lfloor k/2\rfloor-1\\
gn+n_{*}+k & n^{*}+(\lfloor k/2\rfloor-1)g+1 \leq n\leq kg+\mathbf{1}_{\{k>a/2\}}.
\end{cases}
\]
\end{lemma}
\begin{proof}
By Lemma \ref{lemma:bounds}
and recalling that $\rho_1+\rho_2+\rho_3=a$
and $\rho_1\rho_2\rho_3=1$,
we can write
$\ell_{23}(n) = gn+n_{*}+ \xi$, where
\[
\xi = \xi(a,n,k) = 1+\frac{n}{a+1+\delta}-\frac{k}{a+2+\delta}.
\]
Here $\delta\in(\delta_{-},\delta_{+})$,
with $\delta_{\pm}$ as described in \eqref{rho3bounds}.
Let $n\in[0,n_*]$ and let us show that in this range $\xi\in(0,1)$. We can write
\[
0<\xi<1
\impliedby
\begin{cases}
k<a+2+\delta,\\
k<a+1+\delta.
\end{cases}
\]
Since $k\leq a$ by assumption,
the above holds.
Consider now $n\in[n_*+(i-1)g+1,n_*+ig]$
for $i=1,\ldots,\lceil k/2\rceil$, and let us show that $\xi\in(i,i+1)$ in these intervals.
We have
\[
\xi\in(i,i+1)
\impliedby
\begin{cases}
\ds k<\frac{\delta(a+2+\delta)}{\delta(a+2+\delta)-2},\\
k<(1+\delta)(a+2+\delta).\rule{0pt}{15pt}
\end{cases}
\]
The second inequality holds since $k\leq a$ by assumption.
In the first line 
the denominator is positive and the fraction is greater than $a$ for all $\delta\in(\delta_{-},\delta_{+})$,
therefore the inequality holds, again by $k\leq a$.
Next, when $n=n^{*}$ we have
\[
\xi\in(\lceil k/2\rceil,\lceil k/2\rceil+1)
\impliedby
\begin{cases}
\ds k<\frac{2(a+2+\delta)(2(a+1+\delta)-\delta)}{\delta(a+2+\delta)-2},\\
k(2-\delta(a+2+\delta))<0.\rule{0pt}{15pt}
\end{cases}
\]
In the second line the coefficient of $k$ is negative so the inequality holds.
The fraction in the first line is greater
than $a$ for every $\delta\in(\delta_{-},\delta_{+})$,
so the inequality holds.
In the range $n\in[n^{*}+(j-1)g+1,n^{*}+jg]$
with $j=1,\ldots,\lfloor k/2\rfloor-1$, we can have
\[
\xi\in(\lceil k/2\rceil+j,\lceil k/2\rceil+j+1)
\impliedby
\begin{cases}
\ds k<\frac{(1+2\delta)(a+2+\delta)}{\delta(a+2+\delta)-1},\\
k(2-\delta(a+2+\delta))<2\delta(a+2+\delta).\rule{0pt}{15pt}
\end{cases}
\]
Again, the coefficient of $k$ in the second line is negative so the inequality holds trivially.
In the first line the denominator is positive and the fraction is greater than $a$ for all $\delta\in(\delta_{-},\delta_{+})$,
so the inequality holds since we have $k\leq a$ by assumption.
Finally, in the range $n\in[n^{*}+(\lfloor k/2\rfloor-1)g+1\leq n\leq kg+\mathbf{1}_{\{k>a/2\}}]$, we can write
\[
\xi\in(k,k+1)
\impliedby
\begin{cases}
\ds k<\frac{(1+\delta)(a+2+\delta)}{\delta(a+2+\delta)-1},\\
\ds k>\frac{(a+2+\delta)\mathbf{1}_{\{k>a/2\}}}{a+1+\delta+\delta(a+2+\delta)}.\rule{0pt}{18pt}
\end{cases}
\]
The fraction in the first line is greater than $a$ for all $\delta\in(\delta_{-},\delta_{+})$,
so the inequality holds since $k\leq a$.
The fraction in the second line is smaller than $a/2$ for all $\delta\in(\delta_{-},\delta_{+})$,
hence the second inequality holds as well.
\end{proof}

%% SECTION 3: PROOF OF MAIN THEOREM %%
\section{Proof of Theorem \ref{thm2}}\label{S3}

We can now collect the result obtained in
the previous section and prove Theorem~\ref{thm2}.
As we mentioned earlier, we count the number of points
in the triangle by slicing it along vertical lines
and then summing over these lines.
We use the shorthands
\[
\eta=\mathbf{1}_{\{k>a/2\}}
\quad\text{and}\quad
\eps_k = \mathbf{1}_{\{k\equiv 1\,(2)\}} = k-2\lfloor k/2\rfloor.
\]
We can write
\begin{equation}\label{2609:eq003}
\begin{split}
|\{k\mathcal{T}_a\cap \ZZ^2\}|
&=
\sum_{n=-k}^{k(a+1)+\eta} |k\mathcal{T}_a\cap \ZZ^2\cap \{x=n\}|
\\
&=
\sum_{n=-k}^{k(a+1)+\eta} \lfloor\ell_{13}(n)\rfloor
-
\sum_{n=-k}^{-1} \lfloor\ell_{12}(n)\rfloor
-
\sum_{n=0}^{k(a+1)+\eta} \lfloor\ell_{23}(n)\rfloor.
\end{split}
\end{equation}
Let us compute the sums separately. First we
claim that
\begin{equation}\label{2709:eq001}
\begin{split}
\sum_{n=-k}^{k(a+1)+\eta} \lfloor\ell_{13}(n)\rfloor
=&
\frac{1}{2}(k^2a^3+4k^2a^2+9k^2a+ka^2+ka+12k^2+3k)
\\
&+\eta(ka^2+2ka+3k+a)-\frac{\eps_k}{2}.
\end{split}
\end{equation}
To see this, we use Lemma \ref{lemma:topline} to write
\[
\begin{split}
\sum_{n=-k}^{k(a+1)+\eta} \lfloor\ell_{13}(n)\rfloor
&=
\sum_{n=-k}^{k(a+1)+\eta}(an+C)
+
\sum_{i=1}^{\lfloor k/2\rfloor}
    \sum_{n=n_*+(i-1)g+1}^{n_*+ig} \!\!\!\! i
\quad+\quad
\lfloor k/2\rfloor
\\
&\qquad+\;
\sum_{j=1}^{\lceil k/2\rceil-1}
    \sum_{n=n^*+(j-1)g+1}^{n^*+jg} \!\!\!\! \lfloor k/2 \rfloor+j
\quad+\!\!\!\!\!\!\!\!
\sum_{n=n^*+g(\lceil k/2\rceil-1)+1}^{k(a+1)+\eta} \!\!\!\! k,
\end{split}
\]
where $g=a+2$, $C=kg$, $n_*=a+1-2k$ and $n^*=n_*+g\lfloor k/2\rfloor+1$.
Therefore we obtain
\[
\begin{split}
\sum_{n=-k}^{k(a+1)} \lfloor\ell_{13}(n)\rfloor
={}&
\frac{1}{2}(k^2a^3+4k^2a^2+8k^2a+8k^2+(1+\eta)(ka^2+ka)+(4+2\eta)k+a\eta)
\\
&+
\frac{1}{8}(k^2a+2k^2+2ka+4k-\eps_k(2ka+a+4k+2))
+
\frac{k-\eps_k}{2}
\\
&+
\frac{1}{8}(3k^2a+6k^2-6ka-12k+\eps_k(2ka+a+4k+2))
+
k^2+\eta k.
\end{split}
\]
Simplifying the above we obtain \eqref{2709:eq001}.
Concerning the sum with $\ell_{12}$ in \eqref{2609:eq003}
we easily compute, using Lemma \ref{lemma:leftline},
\begin{equation}
\sum_{n=-k}^{-1} \lfloor\ell_{12}(n)\rfloor
=
-\sum_{n=-k}^{-1}n \;\;-\!\!\!\!\!\! \sum_{n=-\lfloor k/2\rfloor}^{-1} \!\!\! 1\;
=
\frac{k^2}{2} + \frac{\eps_k}{2}.
\end{equation}
Next, let us show that
\begin{equation}\label{2709:eq002}
\begin{split}
\sum_{n=0}^{k(a+1)+\eta} \lfloor\ell_{23}(n)\rfloor
=&
\frac{1}{2}(k^2a^3+3k^2a^2+6k^2a+ka^2+ka+2k^2+3k)
\\
&+
\eta(ka^2+2ka+3k+a)+\frac{\eps_k}{2}-1.
\end{split}
\end{equation}
By Lemma \ref{lemma:rightline} we can write
\[
\begin{split}
\sum_{n=0}^{k(a+1)+\eta} \lfloor\ell_{23}(n)\rfloor
={}&
\sum_{n=0}^{k(a+1)+\eta} (gn+n_*)
+
\sum_{i=1}^{\lceil k/2\rceil}
    \sum_{n=n_*+(i-1)g+1}^{n_*+ig} \!\!\!\! i
\quad+\quad
\lceil k/2\rceil
\\
&+
\sum_{j=1}^{\lfloor k/2\rfloor -1}
    \sum_{n=n^*+(j-1)g+1}^{n^*+jg}
        \!\!\!\! \lceil k/2\rceil +j
\qquad+\!\!\!\!\!\!\!\!
\sum_{n=n^*+g(\lfloor k/2\rfloor -1)+1}^{k(a+1)+\eta}
    \!\!\!\!\!\!\!\!\!\!\!\! (n_*+1),
\end{split}
\]
where $g=a+1$, $n_*=k-1$ and
$n^*=n_*+\lceil k/2\rceil g+1$.
Therefore we obtain
\[
\begin{split}
\sum_{n=0}^{k(a+1)+\eta} \lfloor\ell_{23}(n)\rfloor
={}&
\frac{1}{2}(k^2a^3+3k^2a^2+5k^2a+ka^2+3k^2+k-2)
\\
&+
\eta(ka^2+2ka+2k+a)
\\
&+
\frac{1}{8}(k^2a+2ka+k^2+2k+\eps_k(2ka+2k+3a+3))+\frac{k+\eps_k}{2}
\\
&+
\frac{1}{8}(3k^2a-6ka+3k^2-6k-\eps_k(2ka+2k+3a+3))
\\
&-\rule{0pt}{14pt}
k^2+ka+k+\eta k.
\end{split}
\]
Simplifying the terms we obtain \eqref{2709:eq002}.
By inserting \eqref{2709:eq001}--\eqref{2709:eq002}
in \eqref{2609:eq003}, we finally arrive at
\[
|\{(x,y)\in k\mathcal{T}_a\cap \ZZ^2\}|
=
\frac{k^2}{2}(a^2+3a+9)
+
1-\frac{3\eps_k}{2},
\]
which gives Theorem \ref{thm2} after noticing that the last two terms are equal to $\theta(k)$.

%% BIBLIOGRAPHY %%

\end{document}